\documentclass[11pt]{amsart}

\usepackage{enumerate}
\usepackage{amsmath, amssymb, amsthm, amsfonts}
\usepackage{tikz}
\usepackage{geometry}
\usepackage{bbm} 
\usepackage{url}
\usepackage{cmap} 

\usepackage[pdftex, bookmarks=true]{hyperref} 

\usepackage{makecell}  
\usepackage{nicefrac}

\newtheorem{theorem}{Theorem}[section]

\newtheorem{corollary}[theorem]{Corollary}
\newtheorem{lemma}[theorem]{Lemma}

\newtheorem*{claim*}{Claim}

\theoremstyle{definition}
\newtheorem{remark}[theorem]{Remark}
\newtheorem*{remark*}{Remark}

\newtheorem{definition}[theorem]{Definition}

\def\E{\mathbb{E}}

\def\Rb{\mathbb{R}}

\def\Nb{\mathbb{N}}
\def\Gb{\mathbb{G}}

\def\Fc{\mathcal{F}}

\def\Pc{\mathcal{P}}  
\def\dr{\mathrm{d}}   

\def\al{\alpha}

\def\ga{\gamma}
\def\Ga{\Gamma}

\def\la{\lambda}

\newcommand{\defeq}{\mathrel{\vcenter{\baselineskip0.5ex \lineskiplimit0pt
                     \hbox{\scriptsize.}\hbox{\scriptsize.}}}%
                     =}

\def\ind{\mathbbm{1}} 

\title[Improved replica bounds for the independence ratio]{Improved replica bounds for the independence ratio of random regular graphs}

\author{Viktor Harangi}
\thanks{This work was supported by the MTA-R\'enyi Counting in Sparse Graphs ``Momentum'' Research Group, NRDI 
grant KKP 138270, and the Hungarian Academy of Sciences (J\'anos Bolyai Scholarship).}
\address{Alfr\'ed R\'enyi Institute of Mathematics, Budapest, Hungary} 
\email{harangi@renyi.hu}

\begin{document}

\maketitle

\begin{abstract}
Studying independent sets of maximum size is equivalent to considering the hard-core model with the fugacity parameter $\lambda$ tending to infinity. Finding the independence ratio of random $d$-regular graphs for some fixed degree $d$ has received much attention both in random graph theory and in statistical physics.

For $d \geq 20$ the problem is conjectured to exhibit 1-step replica symmetry breaking (1-RSB). The corresponding 1-RSB formula for the independence ratio was confirmed for (very) large $d$ in a breakthrough paper by Ding, Sly, and Sun. Furthermore, the so-called interpolation method shows that this 1-RSB formula is an upper bound for each $d \geq 3$. For $d \leq 19$ this bound is not tight and full-RSB is expected. 

In this work we use numerical optimization to find good substituting parameters for discrete $r$-RSB formulas ($r=2,3,4,5$) to obtain improved rigorous upper bounds for the independence ratio for each degree $3 \leq d \leq 19$. As $r$ grows, these formulas get increasingly complicated and it becomes challenging to compute their numerical values efficiently. Also, the functions to minimize have a large number of local minima, making global optimization a difficult task.
\end{abstract}


\section{Introduction}

This paper is concerned with the independence ratio of random regular graphs. A graph is said to be \emph{regular} if each vertex has the same degree. For a fixed degree $d$, let $\Gb(N,d)$ be a uniform random $d$-regular graph on $N$ vertices. Note that $\Gb(N,d)$ has a ``trivial local structure'' in the sense that with high probability almost all vertices have the same local neighborhood: $\Gb(N,d)$ almost surely \emph{converges locally} to the $d$-regular tree $T_d$ as $N \to \infty$. In statistical physics $T_d$ is also known as the \emph{Bethe lattice}. In fact, M\'ezard and Parisi used the expression Bethe lattice for referring to $\Gb(N,d)$ (see \cite{mezard2001bethe} for example), and proposed to study various models on these random graphs.

In a graph, an \emph{independent set} is a set of vertices, no two of which are adjacent, that is, the induced subgraph has no edges. The \emph{independence ratio} of a graph is the size of its largest independent set normalized by the number of vertices. For any fixed degree $d \geq 3$, the independence ratio of $\Gb(N,d)$ is known 
to converge to some constant $\alpha^\ast_d$ as $N \to \infty$ \cite{bayati2013combinatorial}. Determining $\alpha^\ast_d$ is a major challenge of the area. The cavity method, a non-rigorous statistical physics tool, led to a 1-step replica symmetry breaking (1-RSB) formula for $\alpha^\ast_d$. The authors of \cite{barbier2013hardcore} also argued that the formula may be exact for $d \geq 20$, which is widely believed to be indeed the case. Later this 1-RSB formula was confirmed to be exact for (very) large $d$ in the seminal paper of Ding, Sly, and Sun \cite{ding2016maximum}. 

Lelarge and Oulamara \cite{lelarge2018replica} used the \emph{interpolation method} to rigorously establish\footnote{Their work builds on \cite{franz2003replica,franz2003replica_non-poissonian,panchenko2004bounds}. In fact, \cite{barbier2013hardcore} already says that the 1-RSB formula is known to be an upper bound due to the Franz--Leone--Panchenko--Talagrand theorem.} the 1-RSB formula as an upper bound for every $d \geq 3$. This approach also provides $r$-step RSB bounds for any $r \geq 2$. The problem is that these formulas get increasingly complicated and fully solving the corresponding optimization problems seems to be out of reach. Can we, at least, get an estimate or a bound? 

The parameters of the $r$-RSB bound includes a \emph{functional order parameter} which can be thought of as a measure.\footnote{More precisely, as a measure that is supported on measures that are supported on measures that ...\ (iterated $r$ times).} The optimal measure satisfies a certain self-consistency equation. We cannot hope for an exact solution so the natural instinct is to try to find an approximate solution. In the physics literature an iterative randomized algorithm called \emph{population dynamics} is often used to find an approximate solution in the 1-RSB (and occasionally in the 2-RSB) setting for various models. This sounds like a promising approach but we came to the surprising conclusion that for the hard-core model it may be a better strategy to forget about the equation altogether and search among ``simple'' measures. It seems to be possible to get very close to the global optimum using atomic measures with a moderate number of atoms. Furthermore, when we only have a few atoms, we can tune their weights and locations to a great precision, and this seems to outweigh the advantage of having a more ``delicate'' measure (but being unable to tune it to the same precision).

Moreover, using a small number of atoms means that we can compute the value exactly and the interpolation method ensures that what we get is always a rigorous upper bound. In contrast, population dynamics only gives an estimate for the value of the bound because for large populations one simply cannot compute the corresponding bound precisely and has to settle for an estimate based on a sample. 

Therefore, our approach is that we try to find local minima of the discrete version (corresponding to atomic measures) using a computer. Even this is a formidable challenge as we will see. Table \ref{table:intro_best} shows the best bounds we found via numerical optimization.

\begin{table}[h!] 
\caption{Upper bounds for the asymptotic independence ratio $\alpha^\ast_d$, $3 \leq d \leq 8$. We indicated which $r$-RSB formula yielded our bound. Note that this is not our estimate for the exact $r$-RSB value, but the best upper bound we found.}
\centering
\begin{tabular}{l|c|c|c|c|c|c}
 & $d=3$ & $d=4$ & $d=5$ & $d=6$ & $d=7$ & $d=8$ \\
 \hline
1-RSB & $0.45085966$ & $0.41119457$ & $0.37926817$ & $0.35298455$ & $0.33088436$ & $0.31197257$ \\
\hline
$r$-RSB & $0.45078521$ & $0.41109414$ & $0.37917031$ & $0.35289949$ & $0.33081722$ & $0.31192223$ \\
 & \scriptsize (5-RSB) & \scriptsize (4-RSB) & \scriptsize (4-RSB) & \scriptsize (4-RSB) & \scriptsize (4-RSB) & \scriptsize (3-RSB) 
\end{tabular}
\label{table:intro_best}
\end{table}

These may seem to be small improvements but we actually expect the true values to be fairly close to our new bounds. 
In particular, for $d=3$ it is reasonable to conjecture that the bound is sharp up to at least five decimal digits, that is, $\alpha^\ast_3 = 0.45078...$.

We also have improvements for $9 \leq d \leq 19$. However, as the degree gets closer to the threshold $d \geq 20$ (above which 1-step replica symmetry breaking is believed to be the truth), the 1-RSB bound gets sharper and our improvement gets smaller. For more details, see the tables in the Appendix.

\subsection{Upper bound formulas}

In Section \ref{sec:formulas} we will explain the RSB bounds in detail. Here we only display a few formulas in order to give the reader an idea of the optimization tasks we are faced with.

For comparison, we start with the replica symmetric (RS) bound: for any $\la>0$ and any $x \in [0,1]$ we have 
\begin{equation} \label{eq:rs}
\alpha^\ast_d \log \la \leq \log\big( 1+\la(1-x)^d \big) 
- \frac{d}{2} \log(1-x^2) .
\end{equation}
Here the fugacity parameter $\la$ is the ``reward'' for including a vertex in the independent set, while $x$ can be thought of as the probability that a \emph{cavity}\footnote{Cavities are vertices of degree $d-1$ that are created by deleting a relatively small number of edges from the graph. The point is that this way the symmetry is broken while the independence ratio is essentially unchanged.} vertex is included. Then the right-hand side expresses the change in the \emph{free energy} when adding a star (i.e., connecting a new vertex to $d$ cavities) versus adding $d/2$ edges between cavities. Choosing $\la$ and $x$ optimally leads to the exact same formula as the Bollob\'as bound from 1981 \cite{bollobas1981independence}, which was based on a first moment calculation for the number of independent sets of a given size. Actually, this relatively simple bound is already asymptotically tight: $\big( 2 + o_d(1) \big) \frac{\log d}{d}$, where the asymptotic lower bound is due to Frieze and {\L}uczak \cite{frieze1992independence}.

The 1-RSB bound says that for any $\la_0>1$ and any $q \in [0,1]$:
\begin{equation} \label{eq:1rsb}
\alpha^\ast_d \log(\la_0) \leq \log\big( 1+(\la_0-1)(1-q)^d \big)
- \frac{d}{2} \log\big( 1- (1-1/\la_0) q^2 \big) .
\end{equation}
Choosing $\la_0$ and $q$ optimally leads to an (implicit) formula for $\alpha^\ast_d$. As we mentioned, this 1-RSB bound is conjectured to be sharp for any $d \geq 20$ and known to be sharp for sufficiently large $d$.

Heavy notation would be needed to describe the $r$-step RSB bounds in general. In order to keep the introduction concise, we only give (a discretized version of) the formula for the case $r=2$: for any $\la_0 >1$, $0<m<1$, and any $p_1, \ldots, p_n, q_1, \ldots, q_n \in [0,1]$ with $p_1 + \cdots + p_n = 1$ we have
\begin{multline} \label{eq:2rsb_discrete}
\alpha^\ast_d \, m \log(\la_0) \leq 
\log \sum_{i_1=1}^n \cdots \sum_{i_d=1}^n \, 
\bigg( \prod_{\ell=1}^d p_{i_\ell} \bigg) 
\bigg( 1+(\la_0-1) 
\prod_{\ell=1}^d (1-q_{i_\ell}) \bigg)^m \\ 
- \frac{d}{2} \log \sum_{i_1=1}^n \sum_{i_2=1}^n p_{i_1} p_{i_2} 
\bigg( 1- (1-1/\la_0) q_{i_1} q_{i_2} \bigg)^m .
\end{multline}

The number of parameters for general $r$ is roughly $2 n_1 \ldots n_{r-1}$, where $n_k$ denotes the number of atoms used at the different layers, so the dimension of the parameter space grows exponentially in $r$, see Section \ref{sec:discrete} for details.

\subsection{The case of degree 3}
One can plug any concrete choice of parameter values into \eqref{eq:2rsb_discrete} to get a bound for the independence ratio. To demonstrate the strength of \eqref{eq:2rsb_discrete} even for small $n$, we include here an example for a 2-RSB bound for $d=3$, $n=4$: the values
\begin{align*}
\la_0 &= 19.3 & \quad 
p_1 &= 0.2493 & \quad 
p_2 &= 0.2778 & \quad 
p_3 &= 0.2880 & \quad  
p_4 &= 0.1849\\
m &= 0.557 & \quad 
q_1 &= 0.1184 & \quad 
q_2 &= 0.5947 & \quad  
q_3 &= 0.8876 & \quad  
q_4 &= 0.9827
\end{align*}
give a bound $\alpha^\ast_3 < 0.450789952<0.45079$ that already comfortably beats the currently best bound ($\approx 0.45086$). Table \ref{table:intro_deg3} shows our best bounds for $d=3$.
\begin{table}[h!] 
\caption{The degree $3$ case: our best $r$-RSB bounds for $\alpha^\ast_3$ for $r=2,3,4,5$ compared to previous upper bounds.}
\centering
\begin{tabular}{l|l|l} 
RS/first moment & 
$0.45906$ & Bollob\'as \cite{bollobas1981independence}\\
McKay bound & 
$0.45537$ & McKay \cite{mckay1987independent}\\
1-RSB & 
$0.45085966$ & Lelarge--Oulamara \cite{lelarge2018replica}\\
2-RSB & 
$0.45078994$ & \\
3-RSB & 
$0.45078602$ & \\
4-RSB & 
$0.45078535$ & \\
5-RSB & 
$0.45078521$ & 
\end{tabular}
\label{table:intro_deg3}
\end{table}

As for lower bounds for small $d$, the best results have been achieved by so-called \emph{local algorithms}. Table \ref{table:lower_deg3} lists a few selected works and the obtained bounds for $\alpha^\ast_3$.
\begin{table}[h!] 
\caption{Lower bounds on $\alpha^\ast_3$. Note that the bounds marked with $\dagger$ cannot be considered fully rigorous as they require some kind of computer simulation or estimation.}
\centering
\begin{tabular}{l|l} 
$0.4328$ & Hoppen \cite{hoppen_thesis}\\
$0.4352$ & Kardo{\v{s}}--Kr\'al--Volec \cite{kardos2011fractional}\\
\makecell[l]{$0.4361$ \\ $0.4380^\dagger$} & Cs\'oka--Gerencs\'er--Harangi--Vir\'ag \cite{csoka2015invariant}\\
$0.4375^\dagger$ & Hoppen--Wormald \cite{hoppen2018local}\\
$0.4453^\dagger$ & Cs\'oka \cite{csoka2016independent}
\end{tabular}
\label{table:lower_deg3}
\end{table}

Note that a beautiful result of Rahman and Vir\'ag \cite{rahman2017local}, building on a work of Gamarnik and Sudan \cite{gamarnik2014limits}, says that asymptotically (as $d \to \infty$) local algorithms can only produce independent sets of half the maximum size (over random regular graphs). For small $d$, however, the independence ratio produced by local algorithms may be the same as (or very close to) $\alpha^\ast_d$.

\subsection{Optimization}
We wrote Python/SAGE codes to perform the numerical optimization for the replica bounds.
\begin{itemize}
\item The first task was to efficiently compute the $r$-RSB formulas and their derivatives w.r.t.\ the parameters.
\item Then we used standard algorithms to perform local optimization starting from random points. As the parameter space grows, more attempts are required to find an appropriate starting point leading to a good local optimum.
\item Eventually we start to encounter a \emph{rugged landscape} with a huge number of local minima, where we cannot expect to get close to the global optimum even after trying a large number of starting points. In order to overcome this obstacle, for $d=3$ we used a technique called \emph{basin hopping}. In each step, the algorithm randomly visits a ``nearby'' local minimum, favoring steps to smaller values. This approach led to the discovery of our best bounds for $d=3$. 
\item The smaller the degree $d$, the deeper we could go in the replica hierarchy (i.e., use larger $r$). We could perform the 3-RSB optimization for $d \leq 10$, the 4-RSB optimization for $d \leq 6$, and the 5-RSB optimization for $d=3$. 
\end{itemize}
See Section \ref{sec:impl} for further details about the implementation.

Although the bounds are hard to find, they are easy to check: one simply needs to plug the specific parameter values into the given formulas. We created a website with interactive SAGE codes where the interested reader may check the claimed bounds and even run simple optimizations: \url{https://www.renyi.hu/~harangi/rsb.htm}. Our codes can be found in the public GitHub repository \url{https://github.com/harangi/rsb}.

\subsection{2-RSB in the literature}

As far as we know, there was only one previous attempt to get an estimate for the 2-RSB formula (only for $d=3$). In \cite{barbier2013hardcore} it reads that ``the 2-RSB calculation is [...] somewhat involved and was done in \cite{rivoire2005thesis} [and obtained the value] $0.45076(7)$''. Rivoire's thesis \cite{rivoire2005thesis} indeed reports briefly of a 2-RSB calculation. Note that, since he considers the equivalent vertex-cover problem (concerning the complements of independent sets), we need to subtract his value from $1$ to get our value. On page 113 he writes that using population dynamics he obtained the following estimate: $0.54924 \pm 0.00007$. For our problem this means $0.45076 \pm 0.00007 = [0.45069,0.45083]$. The value $0.45076(7)$ in \cite{barbier2013hardcore} may have come from mistakenly using an error $\pm 0.000007$ instead of $\pm 0.00007$ when citing Rivoire's work. The thesis only provides a short description of how this estimate was obtained. The author refers to it as ``unfactored'' 1-RSB and it seems to be the same as what we call a non-standard 1-RSB in our remarks after Theorem \ref{thm:r-rsb}. If that is indeed the case, then our findings suggest that its true value should actually be around $0.45081$.

\subsection*{Outline of the paper}
In Section \ref{sec:formulas} we present the general replica bounds and their discrete versions that we need to optimize. Section \ref{sec:num_opt} contains details about the numerical optimization. In Section \ref{sec:revisited} we revisit the $r=1$ case and investigate more sophisticated choices for the functional order parameter. The Appendix contains a table listing our best bounds for different values of $d$ and $r$ (Section \ref{sec:table}) and an overview of the interpolation method for the hard-core model over random regular graphs (Section \ref{sec:ip}). 

\section{Replica formulas} \label{sec:formulas}

Originally the cavity method and belief propagation were non-rigorous techniques in statistical physics to predict the free energy of various models. They inspired a large body of rigorous work, and over the years several predictions were confirmed. In particuler, the so-called \emph{interpolation method} has been used with great success to establish rigorous upper bounds on the free energy. 

In the context of the hard-core model over random $d$-regular graphs, the interpolation method was carried out by Lelarge and Oulamara in \cite{lelarge2018replica}, building on the pioneering works \cite{franz2003replica,franz2003replica_non-poissonian,panchenko2004bounds}. First we present the general $r$-step RSB bound obtained this way.

\subsection{The general replica bound} \label{sec:r-rsb}
For a topological space $\Omega$ let $\Pc(\Omega)$ denote the space of Borel probability measures on $\Omega$ equipped with the weak topology. We set $\Pc^1 \defeq \Pc\big( [0,1] \big)$ and then recursively $\Pc^{k+1} \defeq \Pc\big( \Pc^{k} \big)$ for $k \geq 1$. The general bound will have the following parameters:
\begin{itemize}
\item $\la>1$;
\item $0< m_1, \ldots, m_r <1$ corresponding to the so-called Parisi parameters;
\item a measure $\eta^{(r)} \in \Pc^r$.
\end{itemize}
\begin{definition} \label{def:recursive_sampling}
Given a fixed $\eta^{(r)} \in \Pc^r$, we choose (recursively for $k=r-1,r-2,\ldots,1$) a random $\eta^{(k)} \in \Pc^k$ with distribution $\eta^{(k+1)}$. Finally, given $\eta^{(1)}$ we choose a random $x \in [0,1]$ with distribution $\eta^{(1)}$. In fact, we will need $d$ independent copies of this random sequence, indexed by $\ell \in \{1,\ldots, d\}$. Schematically:
\[ \eta^{(r)} \, \to \, \eta_\ell^{(r-1)} \, \to \, \cdots \, \to \, 
\eta_\ell^{(1)} \, \to \, x_\ell \quad (\ell= 1,\ldots, d) .\] 

For $1 \leq k \leq r$ we define $\Fc_k$ as the $\sigma$-algebra generated by 
$\eta_\ell^{(r-1)}, \ldots, \eta_\ell^{(k)}$, $\ell=1,\ldots,d$, and by $\E_k$ we denote the conditional expectation w.r.t.\ $\Fc_k$. Note that $\Fc_r$ is the trivial $\sigma$-algebra and hence $\E_r$ is simply $\E$.

Given a random variable $V$ (depending on the variables $\eta_\ell^{(k)}, x_\ell$), let us perform the following procedure: raise it to power $m_1$, then apply $\E_1$, raise the result to power $m_2$, then apply $\E_2$, and so on. In formula, let $T_0 V \defeq V$ and recursively for $k=1, \ldots, r$ set 
\[ T_k V \defeq \E_k (T_{k-1} V)^{m_k} .\]
In this scenario, applying $\E_k$ means that, given $\eta_\ell^{(k)}$, $\ell=1,\ldots,d$, we take expectation in $\eta_\ell^{(k-1)}$, $\ell=1,\ldots,d$ (or in $x_\ell$ if $k=1$). 
\end{definition}

Now we are ready to state the $r$-RSB bound given by the interpolation method.
\begin{theorem} \label{thm:r-rsb}
Let $r \geq 1$ be a positive integer and $\la, m_1, \ldots, m_r, \eta^{(r)}$ parameters as described above. Let $x_\ell$, $\ell=1,\ldots,d$ denote the random variables obtained from $\eta^{(r)}$ via the procedure in Definition \ref{def:recursive_sampling}. Then we have the following upper bound for the asymptotic independence ratio $\alpha^\ast_d$ of random $d$-regular graphs:
\begin{equation*}
\alpha^\ast_d \, m_1 \cdots m_r \log \la 
\leq \log T_r \big( 1+ \la(1-x_1)\cdots(1-x_d) \big) 
- \frac{d}{2} \log T_r (1-x_1 x_2) .
\end{equation*}
\end{theorem}
This was rigorously proved in \cite{lelarge2018replica}. They actually considered a more general setting incorporating a class of (random) models over a general class of random hypergraphs (with given degree distributions). They used the hard-core model over $d$-regular graphs as their chief example, working out the specific formulas corresponding to their general RS and 1-RSB bounds. Theorem \ref{thm:r-rsb} follows from their general $r$-RSB bound \cite[Theorem 3]{lelarge2018replica} exactly the same way as in the RS and 1-RSB case.

We should make a number of remarks at this point.
\begin{itemize}
\item Above we slightly deviated from the standard notation as the usual form of the Parisi parameters would be
\[ 0 < \hat{m}_1 < \cdots < \hat{m}_r < 1 ,\]
where $\hat{m}_k$ can be expressed in terms of our parameters $m_k$ as follows:
\[ \hat{m}_r = m_1; \quad 
\hat{m}_{r-1} = m_1 m_2; \quad 
\ldots; \quad 
\hat{m}_{1} = m_1 m_2 \cdots m_r .\]
As a consequence, the indexing of $\Fc_k$, $\E_k$, $T_k$ is in reverse order, and the definition of $T_k$ simplifies a little because raising to power $1/\hat{m}_{r-k+2}$ and then immediately to $\hat{m}_{r-k+1}$ (as done, for example, in \cite{panchenko2004bounds}) amounts to a single exponent $\hat{m}_{r-k+1}/\hat{m}_{r-k+2}=m_k$ in our setting.

\item Also, generally there is an extra layer of randomness (starting from an $\eta^{(r+1)} \in \Pc^{r+1}$) resulting in another expectation outside the $\log$. This random choice is meant to capture the local structure of the graph in a given direction. However, when the underlying graph is $d$-regular (meaning that essentially all vertices see the same graph structure locally), we do not need this layer of randomness (in principle). Therefore, in the $d$-regular case one normally chooses a trivial $\eta^{(r+1)} = \delta_{\eta^{(r)}}$. That is why we omitted $\eta^{(r+1)}$ and started with a deterministic $\eta^{(r)}$.

For $d \geq 20$, where the $1$-RSB bound is (conjectured to be) tight, the optimal choice of parameters indeed uses a trivial $\eta^{(r+1)} = \delta_{\eta^{(r)}}$ with $r$ being $1$ in this case.

For $d \leq 19$, the same choice gives us a 1-RSB upper bound (which is not tight any more). Let us call this the \emph{standard 1-RSB bound}, and, in general, we call an $r$-RSB bound \emph{standard} if it was obtained by using a deterministic $\eta^{(r)}$ at the start. Then a non-standard bound would use $\eta^{(r+1)}$ (and hence random $\eta_\ell^{(r)}$ variables). Note that a non-standard $r$-RSB bound is actually a special case of standard $(r+1)$-RSB bounds in the limit $m_{r+1} \to 0$. So even though it is possible to improve on standard $r$-step bounds by non-standard $r$-step bounds, it actually makes more sense to use the extra layer to move to $(r+1)$-step bounds instead (and use some positive $m_{r+1})$. 

\item The full RSB picture is well-understood for the famous Sherrington--Kirkpatrick model \cite{talagrand2006parisi,panchenko2013sk}, where the infimum of the $r$-RSB bound converges to the free energy as $r \to \infty$. It is reasonable to conjecture that this is the case for the hard-core model as well. There is some progress towards this in \cite{cojaoghlan2019spin} where a variational formula is obtained for $\al_d^\ast$.
\end{itemize}

\subsection{A specific choice} \label{sec:choice}
The formula in Theorem \ref{thm:r-rsb} would be hard to work with numerically because it would only give good results for very large $\la$. So we make a specific choice (similar to the one made in \cite[Section 3.2.1]{lelarge2018replica} in the case $r=1$) that may not be optimal but will allow us to use numerical optimization. We consider the limit $\la \to \infty$ and $m_1 \to 0$ in a way that $m_1 \log \la$ stays constant and $x$ is concentrated on the two-element set $\{0, 1-1/\la\}$, meaning that $\eta^{(1)}$ is a distribution $q \delta_{1-1/\la} + (1-q) \delta_0$ for some random $q \in [0,1]$.

For a fixed $\la_0>1$ let $\log \la_0 = m_1 \log \la$. First we focus on the expressions $T_1 \big( 1+ \la(1-x_1)\cdots(1-x_d) \big)$ and $T_1 (1-x_1 x_2)$. If each $x_\ell \in \{0, 1-1/\la\}$ was fixed, we would have the following in the limit as $\la \to \infty$, $m_1 \to 0$ with $m_1 \log \la = \log \la_0$:
\begin{align*}
\big( 1+ \la(1-x_1)\cdots(1-x_d) \big)^{m_1} &\to 
\begin{cases}
\la_0 & \mbox{ if each $x_\ell$ is $0$;}\\
1 & \mbox{ otherwise;} 
\end{cases} \\
(1-x_1 x_2)^{m_1} &\to 
\begin{cases}
1/\la_0 & \mbox{ if } x_1=x_2=1-1/\la ;\\
1 & \mbox{ otherwise.} 
\end{cases} 
\end{align*}
Therefore, conditioned on 
\[ \eta_\ell^{(1)} = q_\ell \delta_{1-1/\la} + (1-q_\ell) \delta_0 \]
for some deterministic $q_1, \ldots, q_d \in [0,1]$, we get 
\begin{align*}
T_1 \big( 1+ \la(1-x_1)\cdots(1-x_d) \big) &\to  
1+(\la_0-1) (1-q_1)\cdots(1-q_d) ;\\
T_1 (1-x_1 x_2) &\to 
1-(1-1/\la_0) q_1 q_2 .
\end{align*}
In the resulting formula the randomness in layer $1$ disappears along with the Parisi parameter $m_1$. After re-indexing ($k \to k-1$) we get the following corollary.
\begin{corollary} \label{cor:r-rsb}
Let $\la_0>1$ and $0< m_1, \ldots, m_{r-1} <1$. Furthermore, fix a deterministic $\pi^{(r-1)} \in \Pc^{r-1}$ and take $d$ independent copies of recursive sampling:
\[ \pi^{(r-1)} \, \to \, \pi_\ell^{(r-2)} \, \to \, \cdots \, \to \, 
\pi_\ell^{(1)} \, \to \, q_\ell \quad (\ell= 1,\ldots, d) .\] 
We define the conditional expectations $\E_k$ and the corresponding $T_k$ as before, w.r.t.~this new system of random variables. Then 
\begin{equation*}
\alpha^\ast_d \, m_1 \cdots m_{r-1} \log \la_0 
\leq \log T_{r-1} \big( 1+ (\la_0-1)(1-q_1)\cdots(1-q_d) \big) 
- \frac{d}{2} \log T_{r-1} \big( 1-(1-1/\la_0) q_1 q_2 \big) .
\end{equation*}
\end{corollary}
\begin{proof}
For a formal proof one needs to define an $\eta^{(r)}=\eta_\la^{(r)} \in \Pc^r$ for the fixed $\pi^{(r-1)}$ and any given $\la$ such that the corresponding $\eta_\ell^{(1)}$ is distributed as $q_\ell \delta_{1-1/\la} + (1-q_\ell) \delta_0$. Then Theorem \ref{thm:r-rsb} can be applied and we get the new formula in the limit. 
\end{proof}

\subsection{Discrete versions} \label{sec:discrete}

In our numerical computations we will use the bound of Corollary \ref{cor:r-rsb} in the special case when each distribution is discrete. 

For $r=1$ we have a deterministic $q$ and we get back \eqref{eq:1rsb}, while $r=2$ gives \eqref{eq:2rsb_discrete}. 

Let $r=3$. For any $\la_0 >1$, $0<m_1,m_2<1$, $p_i \geq 0$ with $\sum p_i=1$, $p_{i,j} \geq 0$ with $\sum_j p_{i,j}=1$ for every fixed $i$, and $q_{i,j} \in [0,1]$ we get that
\begin{align*} 
& \alpha^\ast_d \, m_1 m_2 \log(\la_0) \leq \log R^{\mathrm{star}} 
- \frac{d}{2} \log R^{\mathrm{edge}} \mbox{, where} \\
& R^{\mathrm{star}} = \sum_{i_1} \cdots \sum_{i_d} \, 
\bigg( \prod_{\ell=1}^d p_{i_\ell} \bigg) 
\left( \sum_{j_1} \cdots \sum_{j_d} \,
\bigg( \prod_{\ell=1}^d p_{i_\ell,j_\ell} \bigg) 
\bigg( 1+(\la_0-1) 
\prod_{\ell=1}^d (1-q_{i_\ell,j_\ell}) \bigg)^{m_1} \right)^{m_2} ;\\ 
& R^{\mathrm{edge}} = \sum_{i_1} \sum_{i_2} 
p_{i_1} p_{i_2}
\left( \sum_{j_1} \sum_{j_2} 
p_{i_1,j_1} p_{i_2,j_2} 
\bigg( 1-(1-1/\la_0) 
q_{i_1,j_1} q_{i_2,j_2} \bigg)^{m_1} \right)^{m_2} .
\end{align*}

For a general $r \geq 1$, we will index our parameters $p_s,q_s$ with sequences $s=\big( s^{(1)},\ldots,s^{(k)} \big)$ of length $|s|=k \leq r-1$. We denote the empty sequence (of length $0$) by $\emptyset$. Furthermore, we write $s' \succ s$ if $s'$ is obtained by adding an element to the end of $s$, that is, $|s'|=|s|+1$ and the first $|s|$ elements coincide.

Now let $S$ be some set of sequences of length at most $r-1$ such that $\emptyset \in S$. We partition $S$ into two parts $S_{\leq r-2} \cup S_{r-1}$ based on whether the length of the sequence is at most $r-2$ or exactly $r-1$, respectively.

Now the discrete version of the $r$-RSB bound has the following parameters:
\begin{itemize}
\item $\la_0>1$;
\item $0<m_1, \ldots, m_{r-1}<1$;
\item $p_{s} \geq 0$, $s \in S$, satisfying
\[ \sum_{s' \succ s} p_{s'} = 1  \mbox{ for each } s \in S_{\leq r-2} ;\] 
\item $q_{s} \in [0,1]$, $s \in S_{r-1}$. 
\end{itemize}
Now we define the distribution $\pi^{(r-1)} \in \Pc^{r-1}$ corresponding to the parameters $p_s, q_s$. Set 
\[ \pi_s \defeq q_s \in [0,1] \mbox{ for any } s \in S_{r-1} ,\]
and then, recursively for $k=r-2,r-3,\ldots,1,0$, for a sequence $s$ of length $|s|=k$ let
\[ \pi_{s} \defeq \sum_{s' \succ s} p_{s'}\delta_{\pi_{s'}} \in \Pc^{r-1-k} .\] 

We want to use Corollary \ref{cor:r-rsb} with $\pi^{(r-1)} \defeq \pi_{\emptyset} \in \Pc^{r-1}$. The obtained bound can be expressed as follows.

For any $d$-tuple $s_1,\ldots,s_d$ of sequences of length $r-1$, set 
\begin{equation} \label{eq:star_r-1}
R^{\mathrm{star}}_{s_1,\ldots,s_d} \defeq 1+ (\la_0-1)(1-q_{s'_1})\cdots(1-q_{s'_d}) ,
\end{equation}
and then, recursively for $k=r-1, r-2, \ldots, 1$, for any $d$-tuple $s_1,\ldots,s_d$ of sequences of length $k-1$ let
\begin{equation} \label{eq:star_k-1}
R^{\mathrm{star}}_{s_1,\ldots,s_d} \defeq \sum_{s'_1 \succ s_1} \cdots \sum_{s'_d \succ s_d} 
p_{s'_1} \cdots p_{s'_d} 
\big( R^{\mathrm{star}}_{s'_1,\ldots,s'_d} \big)^{m_{r-k}} .
\end{equation}
Similarly, for any pair $s_1,s_2$ of sequences of length $r-1$, set 
\[ R^{\mathrm{edge}}_{s_1,s_2} \defeq 1- (1-1/\la_0)q_{s'_1}q_{s'_2} ,\]
and then, recursively for $k=r-1, r-2, \ldots, 1$. for any pair $s_1,s_2$ of sequences of length $k-1$, let 
\[ R^{\mathrm{edge}}_{s_1,s_2} \defeq 
\sum_{s'_1 \succ s_1} \sum_{s'_2 \succ s_2} 
p_{s'_1} p_{s'_2} 
\big( R^{\mathrm{edge}}_{s'_1,s'_2} \big)^{m_{r-k}} .\]
Then the bound is 
\begin{equation} \label{eq:R_star_edge_bound}
\alpha^\ast_d \, m_1 \ldots m_{r-1} \log(\la_0) \leq \log R_{\emptyset,\ldots,\emptyset}^{\mathrm{star}} 
- \frac{d}{2} \log R_{\emptyset,\emptyset}^{\mathrm{edge}} .
\end{equation}
\begin{remark} \label{rem:par_space}
Normally we fix integers $n_1, \ldots, n_{r-1} \geq 2$ and assume that the $k$-th elements of our sequences come from the set $\{1,\ldots, n_k\}$. This way the number of free parameters (after taking the sum restrictions on the parameters $p_s$ into account) is 
\begin{equation} \label{eq:param_dim} 
(r-1) + 2 n_1 n_2 \cdots n_{r-1} .
\end{equation}
In the tables of Section \ref{sec:num} and the Appendix we will refer to such a parameter space as $[n_1, \ldots, n_{r-1}]$. 
\end{remark}
%

\section{Numerical optimization} \label{sec:num_opt}

\subsection{Numerical results} \label{sec:num}

Our starting point was the observation in \cite{barbier2013hardcore} that the 1-RSB formula for $\alpha^\ast_d$ ``is stable towards more steps of replica symmetry breaking'' only for $d \geq 20$, so it should not be exact for $d \leq 19$. Therefore the 2-RSB bound in Corollary \ref{cor:r-rsb} ought to provide an improved upper bound for some choice of $\la_0,m_1,\pi^{(1)}$. The optimal $\pi^{(1)}$ may be continuous. Can we achieve significant improvement on the 1-RSB bound even by using some atomic measure $\pi^{(1)} = \sum_{i=1}^n p_i \delta_{q_i}$? In other words, can we find good substituting values for the parameters $p_i,q_i$ of the discrete version \eqref{eq:2rsb_discrete} using numerical optimization? We were skeptical because we may not be able to use a large enough $n$ to get a good atomic approximation of the optimal $\pi^{(1)}$. Surprisingly, based on our findings it appears that even a small number of atoms may yield close-to-optimal bounds. Table \ref{table:deg3_2rsb} shows our best 2-RSB bounds for $d=3$ and for different values of $n$.

\begin{table}[h!] 
\caption{Our 2-RSB bounds for $\alpha^\ast_3$ using $n$ atoms.}
\centering
\begin{tabular}{rl} 
$[n]$ & 2-RSB bound \\
\hline
$[2]$  & $0.45080997599102$ \\
$[3]$  & $0.45079057802543$ \\
$[4]$  & $0.45078995066987$ \\
$[5]$  & $0.45078993616987$ \\
$[6]$  & $0.45078993583363$ \\
$[7]$  & $0.45078993582594$ \\
$[11]$ & $0.45078993582525$ \\
$[32]$ & $0.45078993582510$ 
\end{tabular}
\label{table:deg3_2rsb}
\end{table}

Of course, we do not know what the true infimum of the bound in Corollary \ref{cor:r-rsb} is, but our bounds seem to stabilize very quickly as we increase the number of atoms ($n$). Also, we experimented with various other approaches that would allow for better approximations of continuous distributions and they all pointed to the direction that the bounds in Table \ref{table:deg3_2rsb} are close to optimal.

This actually gave us the hope that it may not be impossible to get further improvements by considering $r$-step replica bounds for $r \geq 3$ even though the number of atoms we can use at each layer is indeed very small due to computational capacities. Table \ref{table:deg3_345rsb} shows some bounds we obtained for $d=3$ and $r \geq 3$ using different parameter spaces (see Remark \ref{rem:par_space}).

\begin{table}[h!] 
\caption{Our 3,4,5-step RSB bounds for $\alpha^\ast_3$.}
\centering
\begin{tabular}{lll} 
$r$ & $[n_1,\ldots,n_{r-1}]$ & RSB bound \\
\hline
3 & $[5,4]$     & $0.45078601768$ \\
~ & $[8,3]$     & $0.45078601734$ \\
~ & $[8,4]$     & $0.45078601720$ \\
\hline
4 & $[6,2,2]$   & $0.45078537162$ \\
~ & $[5,3,2]$   & $0.45078534630$ \\
~ & $[8,2,2]$   & $0.45078534531$ \\ 
\hline
5 & $[4,2,2,2]$ & $0.45078520944$ 
\end{tabular}
\label{table:deg3_345rsb}
\end{table}

The dimension of the parameter space \eqref{eq:param_dim} depends only on $r,n_1,\ldots,n_{r-1}$ and not on the degree $d$. However, as we increase $d$, computing $R^{\mathrm{star}}$ (see Section \ref{sec:discrete}) and its derivative takes longer and we have to settle for using smaller values of $r$ and $n_k$. At the same time, the 1-RSB formula is presumably getting closer to the truth as we are approaching the phase transition between $d=19$ and $d=20$. Nevertheless, we tried to achieve as much improvement as we could for each degree $d=3,\ldots, 19$. See the Appendix for results for $d \geq 4$.

\subsection{Implementation} \label{sec:impl}

\subsubsection{Efficient computation}

According to \eqref{eq:R_star_edge_bound}, our RSB upper bound for $\al^\ast_d$ reads as
\begin{equation} \label{eq:bound_func}
\frac{\log R_{\emptyset,\ldots,\emptyset}^{\mathrm{star}} 
- \frac{d}{2} \log R_{\emptyset,\emptyset}^{\mathrm{edge}}}{m_1 \ldots m_{r-1} \log(\la_0)} ,
\end{equation}
where $R_{\emptyset,\ldots,\emptyset}^{\mathrm{star}}$ and $R_{\emptyset,\emptyset}^{\mathrm{edge}}$ were defined recursively through $r-1$ steps, each step involving a multifold summation, see Section \ref{sec:discrete} for details. So our task is to minimize \eqref{eq:bound_func} as a function of the parameters. During optimization the function and its partial derivatives need to be evaluated at a large number of locations. So it was crucial for us to design program codes that compute them efficiently. Instead of trying to do the summations using \texttt{for} loops, the idea is to utilize the powerful array manipulation tools of the Python library NumPy. In particular, one can efficiently perform element-wise calculations or block summations on the multidimensional arrays of NumPy. 

First we show how $R_{\emptyset,\ldots,\emptyset}^{\mathrm{star}}$ can be obtained using such tools. Recall that $S_k$ contains sequences of length $k$ and we have a parameter $p_s$ for any $s \in S_1 \cup \cdots \cup S_{r-1}$ and $q_s$ for any $s \in S_{r-1}$. In particular, in the $[n_1, \ldots, n_{r-1}]$ setup we have $|S_k|=n_1 \cdots n_k$. We do the following steps.
\begin{itemize}
\item $v_k$: vector of length $|S_k|$ consisting of $p_s$, $s \in S_k$ \quad ($1 \leq k \leq r-1$).
\item $P_k$: $d$-dimensional array of size $|S_k| \times \cdots \times |S_k|$ obtained by ``multiplying'' $d$ copies of $v_k$. (Each element of $P_k$ is a product $p_{s_1}\cdots p_{s_d}$ for some $s_1,\ldots,s_d \in S_k$.)
\item $M_{r-1}$: $d$-dimensional array of size $|S_{r-1}| \times \cdots \times |S_{r-1}|$ obtained by ``multiplying'' $d$ copies of the vector consisting of $1-q_s$, $s \in S_{r-1}$, then multiply each element by $\la_0-1$ and add $1$; cf.~\eqref{eq:star_r-1}. 
\item Then, recursively for $k=r-1,r-2, \ldots,1$, given the $|S_k| \times \cdots \times |S_k|$ array $M_k$ we obtain $M_{k-1}$ as follows: we raise $M_k$ to the power of $m_{r-k}$ and multiply by $P_k$ (both element-wise), and perform a block summation: in the $[n_1,\ldots,n_{r-1}]$ setup we divide the array into $n_k \times \cdots \times n_k$ blocks and replace each with the sum of the elements in the block; cf.~\eqref{eq:star_k-1}. 
\item At the end $M_0$ will have a single element equal to $R_{\emptyset,\ldots,\emptyset}^{\mathrm{star}}$.
\end{itemize}
One can compute $R_{\emptyset,\ldots,\emptyset}^{\mathrm{edge}}$ similarly, using two-dimensional arrays this time.

Note that during the computation of $R_{\emptyset,\ldots,\emptyset}^{\mathrm{star}}$ all the $d$-dimensional arrays are invariant under any permutation of the $d$ axes. This means that the same products appear in many instances, hence the same calculations are repeated many times in the approach above. However, typically we get plenty of compensation in efficiency due to the fact that all the calculations can be done in one sweep using powerful array tools. Nevertheless, when the degree $d$ gets above $7$, we do use another approach in the 2-RSB setting $d \, \, [n]$. In advance, we create a list containing all partitions of $d$ into the sum of $n$ nonnegative integers $d=a_1+\cdots+a_n$. We also store the corresponding multinomial coefficients $\binom{d}{a_1,\ldots,a_n}$ in a vector. Then, at each function call, we go through the list of partitions and compute 
\[ p_1^{a_1}\cdots p_n^{a_n} \big(1 + (\la_0-1) (1-q_1)^{a_1}\cdots (1-q_n)^{a_n} \big)^{m_1} ,\]
storing the values in a vector. Then we simply need to take the dot product with the precalculated vector containing the multinomial coefficients.

In both approaches computing the partial derivatives with respect to the parameters ($\la$, $m_k$, $p_s$, $q_s$) is more involved but can be done using similar techniques (array manipulations and partitioning, respectively). As an example, we show how we can compute $\partial R_{\emptyset,\ldots,\emptyset}^{\mathrm{star}} / \partial p_s$ in the first approach. For a given $1 \leq \ell \leq r-1$ we will do this for all $s \in S_\ell$ at once, resulting in a vector of length $|S_\ell|$ consisting of the partial derivatives w.r.t.\ each $p_s$, $s \in S_\ell$. We will use again the vectors $v_k$ and the arrays $P_k,M_k$ obtained during the computation of $R_{\emptyset,\ldots,\emptyset}^{\mathrm{star}}$.
\begin{itemize}
\item $P'_k$: $d$-dimensional array of size $|S_k| \times \cdots \times |S_k|$ obtained by ``multiplying'' the all-ones vector of length $|S_k|$ and $d-1$ copies of $v_k$. 
\item $D_k$: $d$-dimensional array of size $|S_k| \times \cdots \times |S_k|$ obtained by (element-wise) raising $M_k$ to the power of $m_{r-k}-1$ and multiplying by $m_{r-k}$ and by $P_k$.
\item For a given $1 \leq \ell \leq r-1$ we start with $M_\ell$, raise it to the power of $m_{r-\ell}$ and multiply it by $P'_\ell$ (both element-wise) and perform a block summation for blocks of size $1 \times n_\ell \times \cdots \times n_\ell$, resulting in an $|S_\ell| \times |S_{\ell-1}| \times \cdots \times |S_{\ell-1}|$ array that we denote by $M'_{\ell-1}$.
\item Then, recursively for $k=\ell-1,\ell-2, \ldots,1$, given the $|S_\ell| \times |S_k| \times \cdots \times |S_k|$ array $M'_k$ we obtain $M'_{k-1}$ as follows: we ``stretch'' $D_k$ so that its first axis has length $|S_\ell|$ by repeating each element $|S_\ell|/|S_k|$ times (along that first axis) to get an $|S_\ell| \times |S_k| \times \cdots \times |S_k|$ array, which we multiply element-wise by $M'_k$, and perform a block summation for blocks of size $1 \times n_k \times \cdots \times n_k$.
\item At the end we get the array $M'_0$ of size $|S_\ell| \times 1 \times \cdots \times 1$. We simply need to multiply its elements by $d$ to get the partial derivatives w.r.t.\ $p_s$, $s \in S_\ell$.
\end{itemize}

\subsubsection{Local optimization}

Given a differentiable multivariate function, \emph{gradient descent} means that at each step me move in the opposite direction of the gradient at the current point, and thus (hopefully) converging to a local minimum of the function. This is a very natural strategy because we have the steepest initial descent in that direction. There are other standard iterative algorithms that also use the gradient (i.e., the vector consisting of the partial derivatives). They can make more sophisticated steps because they take previous gradient evaluations into account as well, resulting in a faster convergence to a local minimum. Since we have complicated functions for which gradient evaluations are computationally expensive, it is important for us to reach a local optimum in as few iterations as possible. Specifically, we used the \emph{conjugate gradient} and the \emph{Broyden--Fletcher--Goldfarb--Shanno algorithm}, which are both implemented in the Python library SciPy. 

With efficient gradient evaluation and fast-converging optimization at our disposal, we were able to find local optima. However, we were surprised to see that, depending on the starting point, these algorithms find a large number of different local minima of the RSB formulas. This is due to our parameterization: we only consider discrete measures with a fixed number of atoms, and the atom locations are included among the parameters.\footnote{Even in models where the Parisi functional is known to be convex, as in the SK model \cite{auffinger2015parisi}, parameterizing with atom locations (as opposed to with the measure itself) changes the notion of convexity and may lead to functions that are far from convex.} (This is what allowed us to tackle the problem numerically but it also makes the function behave somewhat chaotically.) 

It is hard to get a good picture of the behavior of a function of so many variables. To give some idea, in Figure \ref{fig:plots} we plotted the 2-RSB bound for $d=3 \, [n=5]$ over two-dimensional sections. In both cases we chose three local minima and took the plane $H$ going through them and plotted the function over $H$. (Note that the left one appears to have a fourth local minimum. However, it is only a local minimum for the two-dimensional restriction of the function and it can actually be locally improved when we are allowed to use all dimensions.) 
\begin{figure}[ht]
\centering
\includegraphics[width=6cm]{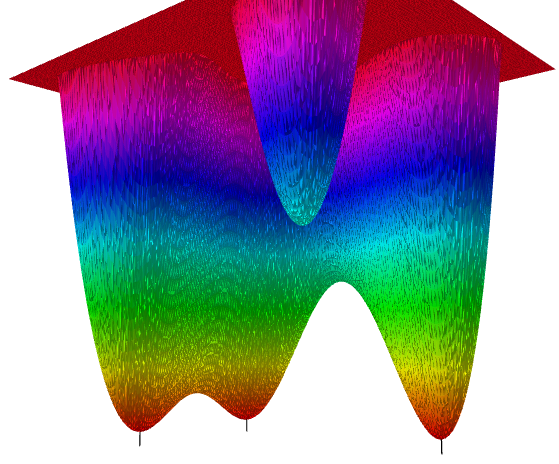} 
\includegraphics[width=6cm]{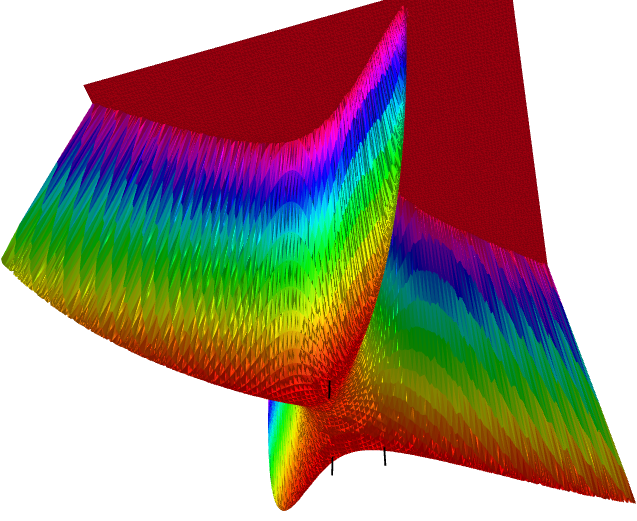} 
\caption{Plots of our 2-RSB bound over two-dimensional sections. The black ticks mark the local minima. We cut the function at a certain height.}
\label{fig:plots}
\end{figure}

Many of these local minima have very similar values. It appears that one would basically need to check them all in order to find the one that happens to be the global minimum (for the given number of atoms). So our strategy is to simply start local optimization from various (random) points to eventually get a strong bound. This seems to work well as long as there are not too many local minima.

\subsubsection{Basin hopping}
As the dimension of the parameter space grows, we start to see a landscape with a huge number of local minima and our chance for picking a good starting point becomes extremely slim. Instead, when we get to a local minimum (i.e., the bottom of a ``basin''), we may try to ``hop out'' of the basin by applying a small perturbation of the variables. After a new round of local optimization, we end up at the bottom of another basin. If the function value decreases compared to the previous basin, we accept this step. If not, then we make a random decision of acceptance/rejection with a probability based on the difference of the values. Such a basin hopping algorithm randomly travels through local minima, with a slight preference for smaller values. (This preference should not be too strong, though, as we have to allow enough leeway for this random travel.) This approach led to the discovery of our best bounds for $d=3$. We mention that in the case of our 5-RSB bound the basin hopping algorithm was running for days.

\subsubsection{Avoiding lower-depth minima}
There is one more subtlety we have to pay attention to, especially when $r \geq 3$. The fact that the $r$-RSB formula contains the $(r-1)$-RSB as a special case means that the optimization has the tendency to converge to such ``lower-depth'' local minima (on the boundary of the parameter space). So it is beneficial to distort the target function in some way in order to force the $r-1$ Parisi parameters to stay away from the boundary. That is, we need to add a penalty term to our function based on the distance of each $m_k$ from $1$. Once the function value is sufficiently small, we can continue the optimization with the original (undistorted) function.

\section{One-step RSB revisited} \label{sec:revisited}

It is possible to improve the previous best bounds even within the framework of the $r=1$ case of the interpolation method. Recall that Theorem \ref{thm:r-rsb} gives the following bound in this case:
\begin{equation*}
\alpha^\ast_d \, m \log \la 
\leq \log \E \big( 1+ \la(1-x_1)\cdots(1-x_d) \big)^m 
- \frac{d}{2} \log \E (1-x_1 x_2)^m ,
\end{equation*}
where $x_1,\ldots,x_d$ are IID from some fixed distribution $\eta$ on $[0,1]$. If we use $\eta=q \delta_{1-1/\la}+(1-q)\delta_0$ and take the limit $m \to 0, \la \to \infty$ with $m \log \la = \la_0$, then we get \eqref{eq:1rsb} as explained in Section \ref{sec:choice} for general $r$. Optimizing  \eqref{eq:1rsb} leads to what we refer to as the 1-RSB bound throughout the paper. In this section we show how one can improve on \eqref{eq:1rsb} for $d \leq 19$ by considering a more sophisticated $\eta$. We will refer to the obtained bounds as $1^+$-RSB bounds. Although this approach is generally inferior to 2-RSB bounds, it is computationally less demanding. In fact, for degrees $d=17,19$ we could only perform the 2-RSB optimization with $n_1=2$ and the obtained bound was actually worse than the $1^+$-RSB bound outlined below.

For the sake of simplicity we start with a choice of $\eta$ only slightly more general than the original one: let $\eta$ have three atoms at the locations
\[ 1-\frac{1}{\lambda^0}=0; \quad  
1-\frac{1}{\lambda^{\nicefrac12}}; \quad 
1-\frac{1}{\lambda^1} .\]
We denote the measures of these atoms by $q_0,q_{\nicefrac12}, q_1 \geq 0$, where $q_0+q_{\nicefrac12}+q_1=1$; i.e., 
\[ \eta 
= q_0\delta_0 
+ q_{\nicefrac12}\delta_{1-1/\sqrt{\la}} 
+ q_1 \delta_{1-1/\la} .\]
Note that the original choice corresponds to the case $q_{\nicefrac12}=0$.

As before, we let $m \to 0, \la \to \infty$ with $m \log \la = \la_0$, which leads to the following bound:
\begin{equation} \label{eq:1plus_half}
\alpha^\ast_d \log(\la_0) \leq \log S - \frac{d}{2} \log E ,
\end{equation}
where
\begin{align*}
S &= 1+\big(\sqrt{\la_0}-1\big)dq_0^{d-1}q_{\nicefrac12} +\big(\la_0-1\big)q_0^d; \\
E &= 1 + \big(1/\sqrt{\la_0}-1 \big)(1-q_0)^2 + \big(1/\la_0-1/\sqrt{\la_0}\big)q_1^2 .
\end{align*}
Substituting $q_1=1-q_0-q_{\nicefrac12}$, we have three remaining parameters: $\la_0, q_0, q_{\nicefrac12}$. Setting the partial derivatives of the right-hand side w.r.t.\ $q_0$ and $q_{\nicefrac12}$ to $0$, we get that 
\[ 0=\partial_{q_0} \big(\log S - \frac{d}{2} \log E \big) = \frac{\partial_{q_0} S}{S} - \frac{d}{2} \frac{\partial_{q_0} E}{E} ,\]
and similarly for $\partial_{q_{\nicefrac12}}$. It follows that for the optimal choice of parameters we have 
\[ \frac{S}{E} 
= \frac{\frac{1}{d} \, \partial_{q_0} S}{\frac{1}{2} \, \partial_{q_0} E} 
= \frac{\frac{1}{d} \, \partial_{q_{\nicefrac12}} S}{\frac{1}{2} \, \partial_{q_{\nicefrac12}} E} .\] 
One can easily compute these partial derivatives to conclude that 
\[ \frac{S}{E} = \frac{q_0^{d-1}}{ \frac{1}{\la_0} (1-q_0-q_{\nicefrac12}) } 
= \frac{(d-1)q_0^{d-2}q_{\nicefrac12} + \big( \sqrt{\la_0} + 1 \big) q_0^{d-1} }
{\frac{1}{\sqrt{\la_0}} (1-q_0) + \frac{1}{\la_0} (1-q_0-q_{\nicefrac12}) } .\]
The second equality gives 
\[ \sqrt{\la_0}=(d-1) \frac{1-q_0-q_{\nicefrac12}}{q_0} ,\] 
which turns the first equality into 
\[ q_{\nicefrac12} = \frac{d-1}{d-2}-\frac{d q_0}{d-1}-\frac{1}{(d-1)(d-2)q_0^{d-2}} . \] 
So our bound has one free parameter left ($q_0$), in which we can easily optimize numerically. 
For $d=3$ one gets $0.450851131$. 
This is the simplest way to improve upon the basic 1-RSB bound. 

More generally, one can take any measure $\tau$ on $[0,1]$ and define $\eta$ as the push-forward of $\tau$ w.r.t.\ the mapping $t \mapsto 1-1/\la^t$. Once again, letting $m \to 0, \la \to \infty$ with $m \log \la = \la_0$, we get the following:
\[
\alpha^\ast_d \log(\la_0) \leq 
\log\bigg( \int \la_0^{\max(0,1-\sum t_\ell)} \, \mathrm{d} \tau^d(t_1,\ldots,t_d) \bigg) \\
- \frac{d}{2} \log\bigg( \int \la_0^{-\min(t_1,t_2)} \, \mathrm{d} \tau^2(t_1,t_2) \bigg) .
\]
For any fixed $\la_0$, an optimal $\tau$ must satisfy a simple fixed point equation involving the convolution power $\tau^{\ast(d-1)}$. For $\mathrm{div} \in \Nb$ one can divide $[0,1]$ into $\mathrm{div}$ many intervals and search among atomic measures $\tau$ with atom locations at $i/\mathrm{div}$, $i=0,1,\ldots,\mathrm{div}$. It is possible to numerically solve the fixed point equation by an iterative algorithm. Then it remains to tune the parameter $\la_0$. We computed these $1^+$-RSB bounds for $\mathrm{div}=1,2,4,8,\ldots,1024$. Note that $\mathrm{div}=1$ corresponds to the original 1-RSB, while $\mathrm{div}=2$ gives \eqref{eq:1plus_half}. Table \ref{table:deg3_1plus_rsb} shows the results for $d=3$.

\begin{table}[h!] 
\caption{Our $1^+$-RSB bounds for $\alpha^\ast_3$. Note that $\mathrm{div}=1$ corresponds to 1-RSB}.
\centering
\begin{tabular}{rl} 
$\mathrm{div}$ & $1^+$-RSB bound \\
\hline
$1$   & $0.45085965358$ \\ 
$2$   & $0.45085113089$ \\ 
$4$   & $0.45084699561$ \\ 
$8$   & $0.45084570075$ \\ 
$16$  & $0.45084535605$ \\ 
$32$  & $0.45084526847$ \\ 
$64$  & $0.45084524648$ \\ 
$128$ & $0.45084524098$ \\ 
$256$ & $0.45084523960$ \\ 
$512$ & $0.45084523926$ \\ 
$1024$& $0.45084523917$ \\ 
\hline
2-RSB & $0.45078993583$
\end{tabular}
\label{table:deg3_1plus_rsb}
\end{table}

\newpage

\section{Appendix: our best bounds} \label{sec:table}

Below we list our best $r$-RSB bounds of $\alpha^\ast_d$ for each degree $3 \leq d \leq 19$ in the following format: 
$\quad r \quad [n_1, \ldots, n_{r-1}] \quad \mbox{bound} \quad$ 
(see Remark \ref{rem:par_space} for the definition of $n_k$). 

For comparison, we included $r=1$, that is, the 1-RSB bound from \cite{lelarge2018replica} that we improve on.

\bigskip

\begin{tabular}{|lll|} 
\multicolumn{3}{c}{\textbf{degree: 3}} \\
\hline
1 & ~           & $0.450859654$ \\
2 & $[32]$      & $0.450789936$ \\
3 & $[8,4]$     & $0.450786018$ \\
4 & $[8,2,2]$   & $0.450785346$ \\ 
5 & $[4,2,2,2]$ & $0.450785210$ \\
\hline
\multicolumn{3}{c}{\textbf{degree: 4}} \\
\hline
1 & ~         & $0.411194564$ \\
2 & $[18]$    & $0.411100755$ \\ 
3 & $[6,4]$   & $0.411095101$ \\
4 & $[4,3,2]$ & $0.411094131$ \\
\hline
\multicolumn{3}{c}{\textbf{degree: 5}} \\
\hline
1 & ~         & $0.379268170$ \\
2 & $[8]$     & $0.379176250$ \\ 
3 & $[3,3]$   & $0.379170372$ \\
4 & $[2,2,3]$ & $0.379170310$ \\
\hline
\multicolumn{3}{c}{\textbf{degree: 6}} \\
\hline
1 & ~         & $0.352984549$ \\
2 & $[7]$     & $0.352905514$ \\ 
3 & $[4,2]$   & $0.352900232$ \\
4 & $[3,2,2]$ & $0.352899485$ \\
\hline
\multicolumn{3}{c}{\textbf{degree: 7}} \\
\hline
1 & ~       & $0.330884354$ \\
2 & $[5]$   & $0.330821477$ \\ 
3 & $[5,2]$ & $0.330817014$ \\
\hline
\multicolumn{3}{c}{\textbf{degree: 8}} \\
\hline
1 & ~       & $0.311972567$ \\
2 & $[6]$   & $0.311925387$ \\ 
3 & $[3,2]$ & $0.311922227$ \\ 
\hline
\multicolumn{3}{c}{\textbf{degree: 9}} \\
\hline
1 & ~       & $0.295553902$ \\
2 & $[5]$   & $0.295520273$ \\ 
3 & $[2,2]$ & $0.295519497$ \\
\hline
\end{tabular}
\quad
\begin{tabular}{|lll|} 
\multicolumn{3}{c}{\textbf{degree: 10}} \\
\hline
1 & ~       & $0.281128003$ \\
2 & $[5]$   & $0.281105186$ \\ 
3 & $[2,2]$ & $0.281104953$ \\ 
\hline
\multicolumn{3}{c}{\textbf{degree: 11}} \\
\hline
1 & ~     & $0.268324856$ \\
2 & $[7]$ & $0.268310124$ \\
\hline
\multicolumn{3}{c}{\textbf{degree: 12}} \\
\hline
1 & ~     & $0.256864221$ \\
2 & $[5]$ & $0.256855205$ \\
\hline
\multicolumn{3}{c}{\textbf{degree: 13}} \\
\hline
1 & ~     & $0.246529415$ \\
2 & $[6]$ & $0.246524236$ \\
\hline
\multicolumn{3}{c}{\textbf{degree: 14}} \\
\hline
1 & ~     & $0.237149865$ \\
2 & $[4]$ & $0.237147193$ \\
\hline
\multicolumn{3}{c}{\textbf{degree: 15}} \\
\hline
1 & ~     & $0.228589175$ \\
2 & $[4]$ & $0.228587914$ \\
\hline
\multicolumn{3}{c}{\textbf{degree: 16}} \\
\hline
1 & ~     & $0.220736776$ \\
2 & $[4]$ & $0.220736278$ \\
\hline
\multicolumn{3}{c}{\textbf{degree: 17}} \\
\hline
1 & ~     & $0.213501935208$ \\
$1^+$ & ~ & $0.213501905193$ \\
\hline
\multicolumn{3}{c}{\textbf{degree: 18}} \\
\hline
1 & ~     & $0.206809394782$ \\
$1^+$ & ~ & $0.206809390398$ \\
2 & $[2]$ & $0.206809390050$ \\
\hline
\multicolumn{3}{c}{\textbf{degree: 19}} \\
\hline
1 & ~     & $0.2005961242697$ \\
$1^+$ & ~ & $0.2005961242567$ \\
\hline
\end{tabular}

\newpage

\section{Appendix: an overview of the interpolation method} \label{sec:ip}

The interpolation method is a rigorous technique to prove upper bounds for the free energy in various models. It has several variants. Originally it was invented by Guerra \cite{guerra2003broken} in the context of the Sherrington--Kirkpatrick spin glass model. In this section we explain the technique for the hard-core model, omitting the technical details and assuming no statistical physics background. We mainly follow the exposition in \cite{ayre2022lower}, where the closely related problem of the chromatic number was considered, and \cite{panchenko2004bounds}.

Given a finite graph $G=(V,E)$, the \emph{partition function} of the \emph{hard-core model} is defined as 
\begin{equation*}
Z_G = Z_{G,\lambda} \defeq \sum_{\sigma \in  \{0,1\}^V} \prod_{v \in V} \lambda^{\sigma_v} 
\prod_{uv \in E} \bigg( 1 - \ind\big( \{ \sigma_u=\sigma_v=1 \} \big) \bigg) ,  
\end{equation*}
where $\lambda>1$ is a parameter often called fugacity. So $Z_G$ counts $0$-$1$ configurations $\sigma=( \sigma_v )_{v \in V}$ on the vertices with no neighboring $1$'s, that is, $I \defeq \{ v \in V \, : \, \sigma_v =1 \}$ is an independent set counted with weight $\la^{|I|}$. Thus $Z_G$ is simply the sum of these weights for all independent sets\footnote{In fact, we should also work with a soft version of $Z$ (at some positive temperature), where neighboring $1$'s are possible but penalized in the partition function. As the temperature goes to zero (i.e., the penalty increases), we get back the hard-core model in the limit. For the sake of simplicity, we describe the interpolation method using the hard-core model but keep in mind that a rigorous treatment would need positive temperatures.}. Let $\alpha(G)$ denote the independence number of $G$ (i.e., the size of the largest independent set). Using the simple inequality 
$Z_{G,\la} \geq \la^{\alpha(G)}$, 
one can bound the independence number as follows:
\begin{equation*} 
\alpha(G) \leq \frac{\log Z_{G,\la}}{\log \la} ,   
\end{equation*}
which is clearly asymptotically tight for any fixed $G$ as $\la \to \infty$.

We are interested in the asymptotic independence ratio $\alpha^\ast_d$ of the random $d$-regular graph $\Gb=\Gb(N,d)$ as the number of vertices $N$ goes to infinity. It follows from the above that for any $\la$ the \emph{normalized free energy} $F_N \defeq \E \log Z / N$ upper bounds $\alpha^\ast_d \log \la$. More precisely, we have 
\begin{equation*} 
\alpha^\ast_d \leq \lim_{N \to \infty} \frac{\E_{\Gb} \log Z_{\Gb,\la}}{N \log \la} . 
\end{equation*}

The method is based on an ``interpolating'' family of models $G_t$, $t \in [0,1]$, with $G_0$ being our original model (plus a disjoint part), which is then ``continuously transformed'' into $G_1$. The key is to prove that the free energy $\E \log Z_{G_t}$ increases as $t$ goes from $0$ to $1$ by showing that the derivative is nonnegative along the way:
\begin{equation} \label{eq:free_energy_diff}
\frac{\partial \E \log Z_{G_t}}{\partial t} \geq 0 
\quad \forall t \in [0,1].
\end{equation}
We will elaborate on this key part of the proof later in Section \ref{sec:mon}. It implies that $\E \log Z_{G_0} \leq \E \log Z_{G_1}$, which will translate to a bound of the form 
\begin{equation} \label{eq:ip_form}
\E \log Z_\Gb \leq \E \log Y - \E \log Y' + o(N) ,
\end{equation}
where $Y$ and $Y'$ are partition functions that are easier to handle. Next we will describe the models in detail. 

\subsection{Variables and factors} 
The models have two types of nodes: \emph{variable nodes} and \emph{fields} (corresponding to \emph{local fields} in physics). We assign a variable $\sigma_v$ to any variable node $v$ that \emph{ranges over} $\{0,1\}$. When we compute the partition function, the sum runs through all possible configurations $\sigma=(\sigma_v)$ with the weight of a configuration being the product of various (penalty and reward) factors. For example, each $\sigma_v=1$ is rewarded with a $\la>1$ factor, while an edge between two variable nodes $v$,$v'$ forbids that $\sigma_v=\sigma_{v'}=1$, i.e., the factor is 
\[ \big( \sigma_v, \sigma_{v'} \big) \mapsto 1-\sigma_v \sigma_{v'} = 
\begin{cases}
1 & \mbox{if } \sigma_v=0 \mbox{ or } \sigma_{v'}=0 ;\\
0 & \mbox{if } \sigma_v=\sigma_{v'}=1 .\\
\end{cases}
\]

A field $u$ does not have a variable, instead there is a probability distribution $\mu_u$ on $\{0,1\}$ assigned to it. In other words, each field $u$ is labelled with a real number $x_u \in [0,1]$ denoting the probability of $1$: $x_u \defeq \mu_u( \{1\})$. 

If there is an edge between a variable node $v$ and a field $u$, then we use the following factor:
\[ \sigma_v \mapsto 1-x_u \sigma_v = 
\begin{cases}
1 & \mbox{if } \sigma_v=0 ;\\
\mu_u( \{0\}) = 1-x_u & \mbox{if } \sigma_v=1 .\\
\end{cases}
\]

Finally, for an edge between two fields $u$,$u'$ we add the following constant factor (that does not depend on $\sigma$):
\[ 1-x_u x_{u'} .\]

\subsection{The models}
Now we are ready to describe the models $G_t$ and the partition functions $Y$ and $Y'$. 
\begin{itemize}
\item Each model $G_t$ has $N$ variable nodes and $dN$ fields. A variable node has $d$ half-edges, each may be connected to another half-edge or to a field.
\item In $G_t$ there are $(1-t)dN/2$ edges connecting two fields, the remaining $tdN$ fields are connected to half-edges of variable nodes randomly, and the remaining $(1-t)dN$ half-edges are matched randomly, creating $(1-t)dN/2$ edges between variable nodes. 
\item In particular, at $t=0$ we get the disjoint union of a random $d$-regular graph $\Gb$ (over the variable nodes) and $dN/2$ pairs of fields, each pair connected by an edge. Therefore 
\[ \log Z_{G_0} = \log Z_\Gb + \log Y' , \]
where $Y'$ is the partition function of the $dN/2$ ``field edges''.
\item At the other endpoint $t=1$, we have $N$ ``stars'', each containing one variable node connected to $d$ fields. We denote the corresponding partition function by $Y$ and hence write 
\[ \log Z_{G_1} = \log Y .\]
\end{itemize}

These are random models. Note that for $Z_\Gb$ the randomness comes purely from the underlying random graph structure, while for $Y$ and $Y'$ it comes from the random labels $x_u$ of the fields $u$ that we will explain next.

In the simplest scenario one fixes a real number $x \in [0,1]$ and use $x_u=x$ for each $u$. In this setup $Y$ and $Y'$ are actually deterministic and can be expressed as products (with the terms corresponding to the $N$ stars and $dN/2$ field edges, respectively): 
\[ Y = \big( 1 + \la (1-x)^d \big)^N \mbox{ and } 
Y' = \big( 1 - x^2 \big)^{dN/2} .\]
(Note that the model of $Y'$ does not have any variable nodes and the ``sum'' is simply the product of constant factors.) Plugging these into \eqref{eq:ip_form} we get back the replica symmetric bound \eqref{eq:rs}. 

More generally, one may choose each $x_u$ independently from a fixed distribution $\nu$ on $[0,1]$. (It is important to use the same $\nu$ for $Y$ and $Y'$.) The resulting partition functions can be factorized again and we get a more general version of the RS bound:
\begin{multline*}
\alpha^\ast_d \log \la \leq 
\int_{[0,1]^d} \log\big( 1 + \la (1-x_1)\cdots(1-x_d) \big) 
\, \dr\nu(x_1) \cdots \dr\nu(x_d) \\
- \frac{d}{2} 
\int_{[0,1]^2} \log\big( 1 - x_1 x_2 \big) \, \dr\nu(x_1) \dr\nu(x_2) .
\end{multline*}

Next we explain how a seemingly insignificant modification of the method turns this approach into a much more powerful tool and resulting in replica symmetry breaking bounds.

\subsection{A weighting scheme}
For a countable index set $\Ga$ let us fix weights $w_\ga \geq 0$, $\ga \in \Ga$, with $\sum_{\ga \in \Ga} w_\ga = 1$ (essentially a probability distribution on $\Ga$) and an arbitrary collection of random variables $\big( x^\ga \big)_{\ga \in \Ga}$, each $x^\ga$ taking values in $[0,1]$. For each $\ga \in \Ga$ we consider a version $Y_\ga$ of $Y$. To this end we need to take independent copies of the collection $\big( x^\ga \big)$ for all fields $u$: 
\[ \big( x^\ga_u \big)_{\ga \in \Ga} 
\mbox{ has the same joint distribution as }
\big( x^\ga \big)_{\ga \in \Ga} .\]
We set the label of each field $u$ to be $x^\ga_u$ and define $Y_\ga$ to be the corresponding partition function. We define $Y'_\ga$ similarly. Then the following weighted version of \eqref{eq:ip_form} is also true:
\begin{equation} \label{eq:ip_weighted}
\E \log Z_\Gb \leq \E \log \sum_{\ga \in \Ga} w_\ga Y_\ga - 
\E \log \sum_{\ga \in \Ga} w_\ga Y'_\ga + o(N) .
\end{equation}
This weighted version is (potentially) more general but it seems that we lose the crucial property of factorization for the formulas inside the $\log$. There is, however, a ``magical'' (random) choice of the coefficients $w_\ga$ (based on the so-called \emph{Derrida--Ruelle cascades}) for which we still have factorization provided that the collection $\big( x^\ga \big)$ is \emph{hierarchically exchangeable}, which notion was introduced in \cite{austin2014hierarchical}.

For a given $r \geq 1$ we use $\Ga=\Nb^r$ as the countable index set. For any fixed parameters $0< m_1, \ldots, m_r < 1$, there exist random weights $w_\ga$ such that for any given $\eta^{(r)}$ of Theorem \ref{thm:r-rsb} we can define the collection $\big( x^\ga \big)$ in a way that \eqref{eq:ip_weighted} yields the bound in the theorem. 

We define $\big( x^\ga \big)$ using the notations of Section \ref{sec:r-rsb}. For any $1 \leq k \leq r$ and any $\ga_1, \ldots, \ga_k \in \Nb$ we will define a random $\eta^{(r-k+1)}(\ga_1,\ldots,\ga_k) \in \Pc^{r-k+1}$. Since we started with a deterministic $\eta^{(r)}$ in Theorem \ref{thm:r-rsb} (see the remarks after the theorem), in our case each $\eta^{(r)}(\ga_1) \defeq \eta^{(r)}$ will be the same for $k=1$. Given $\eta^{(r-k+1)}(\ga_1,\ldots,\ga_k) \in \Pc^{r-k+1}$, we define 
\[ \eta^{(r-k)}(\ga_1,\ldots,\ga_k,\ga_{k+1}), \, \ga_{k+1} \in \Nb, \]
to be conditionally independent and distributed as %
$\eta^{(r-k+1)}(\ga_1,\ldots,\ga_k)$. Finally, for each $\ga=(\ga_1,\ldots,\ga_r) \in \Nb^r$ we sample $x^\ga$ from $\eta^{(1)}(\ga_1,\ldots,\ga_r)$. Schematically:
\[ \eta^{(r)}=\eta^{(r)}(\ga_1) \, \to \, 
\eta^{(r-1)}(\ga_1,\ga_2) \, \to \, 
\cdots \, \to \, 
\eta^{(1)}(\ga_1,\ldots, \ga_r) \, \to \, 
x^\ga .\]
Now suppose that we have a function $f \colon [0,1]^M \to \Rb$. Let us take $M$ independent copies of the above sampling scheme. For each fixed $\ga \in \Nb^r$ we plug the $M$ copies of $x^\ga$ into $f$ resulting in a random variable $V_\ga$. Then one can choose the weights $w_\ga$ randomly in such a way that 
\[ \E \log \sum_{\ga \in \Nb^r} w_\ga V_\ga 
= \E \log T_r \big( V_{(1,\ldots,1)} \big) ,\]
where $T_r$ is defined analogously to Definition \ref{def:recursive_sampling} \cite[Proposition 2]{panchenko2004bounds}.

We will not elaborate on how the weights $w_\ga$ need to be chosen for general $r$. Instead, we focus on the case $r=1$ which already captures the essence of the method. 

\subsection{One-step RSB}

In this case we simply have $\Ga=\Nb$ and each $\eta^{(1)}(\ga)$ is the same deterministic distribution $\eta^{(1)} \in \Pc^1$. In other words, all field labels $x^\ga_u$ are IID across all nodes $u$ in all models $Y_\ga,Y'_\ga$. 
Next we define the random weights $w_\ga$.
\begin{definition}
Given a real number $0<m<1$, let $\hat{w}_1 \geq \hat{w}_2 \geq \ldots$ be the nonincreasing enumeration of the points generated by a nonhomogeneous Poisson point process on $[0,\infty)$ with intensity function $t \mapsto t^{-1-m}$. The sum $\hat{W} \defeq \sum \hat{w}_\ga$ is finite almost surely. For $\ga \in \Nb$ let  
\[ w_\ga \defeq \hat{w}_\ga \big/ \hat{W} .\]
The distribution of $(w_1,w_2,\ldots)$ is called the \emph{Poisson--Dirichlet distribution}.
\end{definition}
In many statistical physics models the relative cluster sizes are believed to behave as the Poisson--Dirichlet distribution for some $m$. It has the following magical property.
\begin{lemma} \cite[Proposition 1]{panchenko2004bounds}
For any fixed $0<m<1$ let $w_\ga$, $\ga \in \Nb$ be the random weights as above. Then for any IID sequence $X_\ga>0$ with $\E X_1^2 < \infty$ we have 
\[ \E \log \, \sum_{\ga=1}^\infty w_\ga X_\ga = \frac{1}{m} \log \E X_1^m .\]
Note that on the left we take expectation both in $w_\ga$ and in $X_\ga$.
\end{lemma}
Applying the lemma for $X_\ga=Y_\ga$ and also for $X_\ga=Y'_\ga$, the bound \eqref{eq:ip_weighted} turns into 
\begin{multline*}
\E \log Z_\Gb 
\leq \frac{1}{m} \log \E Y_1^m - \frac{1}{m} \log \E \big(Y'_1)^m + o(N)\\
= \frac{1}{m} N \log \E \big( 1+ \la(1-x_1)\cdots(1-x_d) \big)^m 
- \frac{1}{m} \frac{dN}{2}\log \E (1-x_1 x_2)^m + o(N),
\end{multline*}
where $x_1,\ldots,x_d$ are IID with distribution $\eta^{(1)}$. Hence we indeed get back Theorem \ref{thm:r-rsb} for $r=1$.

\subsection{Monotonicity of the free energy} \label{sec:mon}
Now we turn to the final ingredient (the reason why all this provides an upper bound): the fact that the free energy of the model $G_t$ is monotone increasing as $t$ goes from $0$ to $1$. In other words, the derivative \eqref{eq:free_energy_diff} is nonnegative.

In $G_t$ there are three types of edges (based on whether there are $0$, $1$, or $2$ variable nodes among the endpoints) and we defined $G_t$ by prescribing the number of edges for all three types. In fact, it is better to define $G_t$ in a way that there is a small portion of the variable nodes with degree $d-1$. Intuitively it is clear that we have to compare the effect (on the free energy) of the addition of an edge of each of the three types. (See \cite[Section 4.2]{ayre2022lower} for an elegant argument justifying this intuition.) 

Suppose that we have any fixed model on $N$ variable nodes with partition function $Z$, where we distinguish some of the nodes as \emph{cavity nodes} (in our setting they belong to the variable nodes that do not have full degree $d$ but only degree $d-1$). The number of cavity nodes should be small compared to $N$ but should converge to $\infty$ as $N \to \infty$. We want to understand the effect (on $\log Z$) of the addition of a new factor to the model. In our case this will be the addition of either one of the three types of edges:
\begin{itemize}
\item We choose two cavity nodes uniformly and independently and add an edge between them: resulting in a random partition function $Z_{\mathrm{cc}}$. 
\item We add two new fields and add an edge between them: resulting in a random partition function $Z_{\mathrm{ff}}$.
\item We choose a cavity node uniformly and connect it to a new random field: resulting in a random partition function $Z_{\mathrm{cf}}$.
\end{itemize}
What we need to prove is that 
\begin{equation} \label{eq:key_ineq}
\big( \E\log Z_{\mathrm{cc}} - \log Z \big) 
+ \big( \E\log Z_{\mathrm{ff}} - \log Z \big) 
- 2 \big( \E\log Z_{\mathrm{cf}} - \log Z \big) \leq 0 .
\end{equation}

To incorporate the Replica Symmetry Breaking scenario we will have an additional variable $\gamma$: let $\Omega=\{0,1\}^N \times \Ga$ where each $\omega=(\sigma_1,\ldots,\sigma_N,\gamma) \in \Omega$ encodes a configuration of $N$ variables $\sigma_i$ and a \emph{state} $\gamma$ ranging over a countable set $\Gamma$. 

Imagine that at a particular stage of the interpolation we see a certain deterministic model. It is actually not important what the model is; the point is that it assigns a weight $\Psi(\omega)$ to each configuration $\omega \in \Omega$. If we normalize these weights with the corresponding partition function $Z=\sum_{\omega \in \Omega} \Psi(\omega)$, then we get a probability distribution on $\Omega$, called the \emph{Boltzmann distribution}. It is a simple fact that adding a new weight factor $\Psi'(\omega)$ to the model changes the free energy $\log Z$ by $\log \E_\omega \Psi'(\omega)$, where 
$\E_\omega$ means taking expectation w.r.t.\ the Boltzmann distribution. It follows that 
\begin{align*}
\E\log Z_{\mathrm{cc}} - \log Z &= 
\E_{c_1,c_2} \log \E_\omega \big( 1 - \sigma_{c_1} \sigma_{c_2} \big) ;\\
\E\log Z_{\mathrm{ff}} - \log Z &= 
\E_{x_1,x_2} \log \E_\omega \big( 1 - x^\ga_1 x^\ga_2 \big) ;\\
\E\log Z_{\mathrm{cf}} - \log Z &= 
\E_{c_1,x_1} \log \E_\omega \big( 1 - \sigma_{c_1} x^\ga_1 \big) ,
\end{align*}
where $c_1,c_2$ are chosen uniformly and independently from the set $C \subseteq \{1,\ldots,N\}$ of cavities, and $x_1=(x^\ga_1)_{\gamma \in \Gamma}$ and $x_2=(x^\ga_2)_{\gamma \in \Gamma}$ are two independent collections of random variables with the same joint distribution. Then \eqref{eq:key_ineq} follows from the following lemma.
\begin{lemma}
Let $X$ and $Y$ be random $\Omega \to [0,1]$ functions with independent copies $X_1,X_2$ and $Y_1,Y_2$, respectively. Then for any random $\omega \in \Omega$ we have 
\begin{multline*}
\E_{X_1,X_2} \log\bigg( 1- \E_\omega X_1(\omega) X_2(\omega) \bigg) 
+ \E_{Y_1,Y_2} \log\bigg( 1- \E_\omega Y_1(\omega) Y_2(\omega) \bigg) \\
\leq 2 \E_{X,Y} \log\bigg( 1- \E_\omega X(\omega) Y(\omega) \bigg) .
\end{multline*}
\end{lemma}
\begin{proof}
Due to the identity 
\[ \log(1-x) = - \sum_{\ell=1}^\infty \frac{x^\ell}{\ell} ,\]
it suffices to show for each $\ell \geq 1$ that 
\begin{equation} \label{eq:ineq_ell}
\E_{X_1,X_2} \bigg( \E_\omega X_1(\omega) X_2(\omega) \bigg)^\ell 
+ \E_{Y_1,Y_2} \bigg( \E_\omega Y_1(\omega) Y_2(\omega) \bigg)^\ell 
- 2 \E_{X,Y} \bigg( \E_\omega X(\omega) Y(\omega) \bigg)^\ell \geq 0 ,
\end{equation}
which can be easily seen to be equivalent to 
\[ \E_{\omega_1,\ldots,\omega_\ell} \bigg( 
\E_{X} \prod_{i=1}^\ell X(\omega_i) - 
\E_{Y} \prod_{i=1}^\ell Y(\omega_i) \bigg)^2 \geq 0 ,\]
where $\omega_1,\ldots,\omega_\ell$ are independent copies of $\omega$.

Indeed, we may rewrite the first term of \eqref{eq:ineq_ell} as
\begin{multline*}
\E_{X_1,X_2} \E_{\omega_1,\ldots,\omega_\ell} 
\prod_{i=1}^\ell X_1(\omega_i) X_2(\omega_i)
= \E_{\omega_1,\ldots,\omega_\ell} 
\bigg( \E_{X_1} \prod_{i=1}^\ell X_1(\omega_i) \bigg) 
\bigg( \E_{X_2} \prod_{i=1}^\ell X_2(\omega_i) \bigg) \\
= \E_{\omega_1,\ldots,\omega_\ell} 
\bigg( \E_{X} \prod_{i=1}^\ell X(\omega_i) \bigg)^2 .
\end{multline*}
Similar manipulations can be carried out for the two other terms.
\end{proof}

\bibliographystyle{alpha}
\bibliography{refs}

\end{document}